\documentclass{amsart}
\usepackage{amsmath, amssymb,epic,graphicx,mathrsfs,enumerate}
\usepackage[all]{xy}

\usepackage{amsthm}
\usepackage{amssymb}
\usepackage{latexsym}
\usepackage{longtable}
\usepackage{epsfig}
\usepackage{amsmath}
\usepackage{hhline}

\usepackage{color}

 \DeclareMathOperator{\frat}{Frat}

\DeclareMathOperator{\End}{End} \DeclareMathOperator{\h}{{H^1}}

\newcommand{\FF}{\mathbb F}

\newtheorem{thm}{Theorem}
\newtheorem{cor}[thm]{Corollary}
 \newtheorem{lemma}[thm]{Lemma}
\newtheorem{prop}[thm]{Proposition}

\numberwithin{equation}{section}

\renewcommand{\footnote}{\endnote}
\newcommand{\ignore}[1]{}\makeglossary
\begin{document}
\bibliographystyle{amsplain}
\title{Invariable generation of prosoluble groups}
\author{Eloisa Detomi and Andrea Lucchini}
\address{
Eloisa Detomi and Andrea Lucchini,\\ Universit\`a degli Studi di Padova,\\  Dipartimento di Matematica,\\ Via Trieste 63, 35121 Padova, Italy}

\begin{abstract} A group $G$ is \emph{invariably generated} by a subset $S$ of $G$
 if   $G=\langle s^{g(s)} \mid s\in S\rangle$  for each choice of $g(s) \in G$, $s \in S$.
Answering two questions posed by Kantor, Lubotzky and Shalev in \cite{ig-infinite}, we prove that
the free prosoluble group of rank $d \ge 2$ cannot be invariably generated by a finite set of elements,
while  the free solvable profinite group of rank $d$ and derived length $l$ is invariably generated by precisely $l(d-1)+1$ elements.
\end{abstract}

\maketitle

\section{Introduction}

Following  \cite{dixon} 
  we say that
a subset $S$ of a group $G$  invariably generates $G$ if   $G=\langle s^{g(s)} \mid s\in S\rangle$  for each choice of $g(s) \in G$, $s \in S$. 
We also say that a group $G$ is invariably generated (IG for short) if $G$ is invariably generated by some subset $S$ of $G$; when $S$ can be chosen to be finite, we say that $G$ is FIG. A group $G$ is IG if and only if it cannot be covered by a union
of conjugates of a proper subgroup, which amount to saying that in every transitive
permutation representation of $G$ on a set with more than one element there is a
fixed-point-free element. Using this characterization, Wiegold \cite{wieg}  proved that the free group on two (or more) letters is not IG.
 Kantor, Lubotzky and Shalev  studied invariable generation in finite and infinite groups. For example in \cite{ig} they proved
that every finite group $G$ is invariably generated by at most $\log_2|G|$ elements. In \cite{ig-infinite} they studied invariable generation of infinite groups, with emphasis on linear groups, proving
that a finitely generated linear group is FIG
 if and only if it is virtually soluble.

\

Let $G$ be a profinite group. Then generation and invariable generation in $G$ are
interpreted topologically.
Just as every finite group is IG, every profinite group G is also IG. Indeed every
proper subgroup of a profinite group $G$ is contained in a maximal open subgroup
$M,$ and, since $M$ has finite index,  $G$ cannot coincide with the union $\cup_{g\in G}M^g.$
On the other hand, finitely generated profinite groups are not necessarily FIG.
In fact by \cite[Proposition 2.5]{ig},  there exist 2-generated finite
groups H with $d_I (H)$ (the minimal number of invariable generators) arbitrarily
large. This implies that the free profinite  of rank $d \ge 2$ is not FIG. 
In \cite{ig-infinite} the following questions are asked: Are finitely generated prosoluble groups FIG?
 Are  finitely generated soluble profinite groups FIG?

We prove that the first question has in general a negative answer:

\begin{thm}\label{main}
The free prosoluble group of rank $d \ge 2$ is not FIG.
\end{thm}

We will deduce Theorem \ref{main} from the following result (see Theorem \ref{2-gen}). Let $G$ be a finite 2-generated soluble group
and let $p$ be the smallest prime divisor of $|G|$. Then either $d_I(G)\geq p$ or there exists a prime $q>p$ such that
$d_I(G)<d_I(C_q\wr G),$ where $C_q \wr G$ is the wreath product with respect to the regular permutation representation of $G$.

\

In contrast, the second question has a positive answer.
  More precisely we can adapt the arguments used in the proof of Theorem \ref{main} to show:

\begin{thm}\label{sol}
Let $F$ be the free soluble profinite group of rank $d$ and derived length $l$.
Then $d_I(F)=l(d-1)+1$.
\end{thm}

Denote by $d(G)$ the smallest cardinality of a generating set of a finitely generate profinite group $G$.
Clearly if $G$ is pronilpotent,  then $d(G)=d_I(G)$. More precisely, by \cite[Proposition 2.4]{ig} a finitely generated profinite
group $G$ is pronilpotent if and only if every generating set of $G$ invariably generates $G.$
But what can we say about the difference $d_I(G)-d(G)$ when $G$ is a prosupersoluble group?
In this case $G/\frat(G)$ is metabelian, so Theorem \ref{sol} implies that $d_I(G)-d(G)\leq d(G)-1$.
 Although supersolubility is a quite strong property
and in particular a metabelian group is not in general supersoluble, the previous estimate is sharp.

\begin{thm}\label{pro-supersoluble}
Let $F$ be the free prosupersoluble group of rank $d$. Then $d_I(F)=2d-1$.
\end{thm}

\section{Preliminaries}

A profinite group is a topological group that is isomorphic
to an inverse limit of finite groups.  The textbooks \cite{ribes-zal}  and \cite{book:wilson} provide a good introduction to
the theory of profinite groups.  In the context of profinite groups,   generation and invariable generation  are interpreted topologically. 
 By a standard argument (see e.g. \cite[Proposition 4.2.1]{book:wilson}) it can be proved that a profinite group $G$
  is invariably generated by $d$ elements if and only if $G/N$ is invariably generated by $d$ elements for every open normal subgroup $N$ of $G$.  Therefore in the following we will mainly work on finite groups.

If $G$ is a finite soluble group, the minimal number of generators for $G$ can be computed in term of the structure of $G$-modules of the chief factors of $G$ with the following formula due to Gasch\"utz 
 \cite{g2}. 
\begin{prop}\label{gen}
Let $G$ be a finite soluble group. For every irreducible $G$-module $V$ define
$r_G(V)=\dim_{\End_{G}(V)}V$,
   set $\theta_G(V)=0$  if $V$ is a trivial $G$-module,  and $\theta_G(V) = 1$ otherwise,
   and let $\delta_G(V)$ be the number of  chief factors $G$-isomorphic to $V$ and complemented in an arbitrary chief series of $G$.
 Then
 $$d(G)= \max_{V} \left( \theta_G(V) +\left\lceil \frac{\delta_G(V) }{r_G(V)}\right\rceil \right)$$
where  $V$ ranges over the set of non $G$-isomorphic complemented chief factors of $G$ and $\lceil x \rceil$ denotes the smallest integer greater or equal to $x$.
\end{prop}

There is no similar formula for the minimal size of the invariable generating sets.
 The best result in this direction is  a criterion we gave in   \cite{PCIG} to decide whether  an  invariable generating set of a group $G$ can be lifted   to an extension over an abelian normal subgroup.
   To formulate this result, we need to recall some notation from \cite{PCIG}.

 Let $G$ be a  finite group acting irreducibly on an elementary abelian finite $p$-group $V$.
For a positive integer $u$ we consider the semidirect product $V^u \rtimes G$: unless otherwise stated, 
 we assume that  the action of $G$ is diagonal on $V^u$, that is, $G$ acts in the same way on each of the $u$ direct factors.
In \cite[Proposition 8]{PCIG} we proved the following. 

\begin{prop}\label{matrici1}
Suppose $G$ acts faithfully and irreducibly on $V$ and $\h(G,V)=0$.
Assume that $g_1,\dots,g_d$ invariably generate $G$.
  There exist some elements $w_1,\dots,w_d\in V^u$ such that $g_1w_1, g_2w_2,\dots,g_d w_d$ invariably generate
  $V^u\rtimes G$ if and only if
$$u\leq \sum_{i=1}^d  \dim_{\End_G(V)}C_V(g_i).$$
\end{prop}

The assumption  $\h(G,V)=0$ in the case of soluble groups is assured by the following unpublished result by Gasch\"utz (see  \cite[Lemma 1]{st}).

\begin{lemma}\label{H1}
 Let $G\neq 1$ be a finite soluble group and let $V$ be an irreducible $G$-module. Then   $\h(G,V)=0$.
\end{lemma}

In the following we will use this straightforward consequence of Proposition \ref{matrici1}.

\begin{cor}\label{matrici}
 Let $G\neq 1$ be a finite soluble group and let $V$ be an irreducible $G$-module.
Assume that
 $x_1, \ldots , x_d$   invariably generate $V^u\rtimes G$, where
 $x_i=v_i g_i$ with  $v_i \in V^u$ and $g_i \in G$.
 Then
   $g_1,\dots,g_d $   invariably generate $G$
and
$$u\leq \sum_{i=1}^d \dim_{\End_{G/C_G(V)}(V)}C_V(g_i).$$
\end{cor}
\begin{proof}
Clearly, $g_1,\dots,g_d $ invariably generate $G$.  Denote by $\overline{g}_i$
  the image of $g_i$ in the quotient group $G/C_G(V)$. By Lemma \ref{H1} and
   Proposition \ref{matrici1}   we have
  $$u\leq \sum_{i=1}^d \dim_{\End_{G/C_G(V)}(V)}C_V(\overline{g}_i).$$
  Since $\dim C_V(\overline{g}_i)= \dim C_V(g_i)$, the result follows.
\end{proof}

\section{Proof of Theorem \ref{main}}

 If $G$ is a finite group, $\pi(G)$ is the set of primes dividing the order of $G$.

\begin{thm}\label{2-gen} Let $G$ be a 2-generated finite soluble group.
Either $d_I(G)\geq \min{\pi(G)}$ or there exists a finite soluble group $H$ having $G$ as an epimorphic image and such that
\begin{itemize}
\item $d(H)=2$;
\item $d_I(H)>d_I(G);$
\item $\min\pi(H)=\min \pi(G).$
\end{itemize}
\end{thm}
\begin{proof} By Dirichlet's theorem on primes in arithmetic progressions, there exists a prime $q$ such that the exponent of $G$ divides $q-1$.
 Let ${\FF}$ be the field of order $q.$
By a result of Brauer (see e.g. \cite[B 5.21]{doerk})
 ${\FF}$ is a splitting field for $G$ so
$$V:={\FF} G= V_1^{n_1}\oplus\cdots \oplus  V_r^{n_r}$$
 where
the $V_j$ are absolutely  irreducible ${\FF} G$-modules no two of which are $G$-isomorphic, and $n_j=\dim_{\FF} V_j.$
Consider the semidirect product $H=V\rtimes G$; note that
$H$ is isomorphic to
$C_q \wr G$ with respect to the regular permutation representation of $G$. By  \cite[Corollary 2.4]{andrea1}, as $C_q$ and $G$ have coprime orders, $d(C_q \wr G)=\max(d(G),d(C_q)+1)=2$.

Clearly $d_I(G)\le d_I(H)$.
 Assume $d_I(G)=d_I(H)=d $.
 By Corollary \ref{matrici} applied to each homomorphic image $V_j^{n_j} \rtimes G$,
  it follows that there exists
an invariable generating set $g_1,\dots,g_d$ of $G$ such that, for any $j$
\begin{equation*}
n_j\leq \sum_{i=1}^d  \dim_{\FF} C_{V_j}(g_i).
\end{equation*}
Multiplying by $n_j$ we get
\begin{equation*}
n_j^2\leq \sum_{i=1}^d  n_j\dim_{\FF} C_{V_j}(g_i).
\end{equation*}
It follows that:
\begin{equation*}
|G|=\sum_{j=1, \ldots, r} n_j^2\leq \sum_{\substack{i=1, \ldots, d\\j=1, \ldots, r}} n_j\dim_{\FF} C_{V_j}(g_i)
=\sum_{i=1, \ldots, d} \dim_{\FF} C_{{\FF}G}(g_i).
\end{equation*}
On the other hand, by Lemma \ref{centraliser} below,
$$\dim_{\FF} C_{{\FF}G}(g_i)=\frac{|G|}{|g_i|}$$
and therefore
\begin{equation*}
1\leq \sum_{i=1}^d \frac{1}{|g_i|}.
\end{equation*}
Since $d=d_I(G)$ we have $g_i\neq 1$ for every $i$, hence $|g_i|\geq p = \min \pi(G).$
 Therefore
\begin{equation*}
1\leq \sum_{i=1}^d  \frac{1}{|g_i|}\leq \frac{d}{p}
\end{equation*}
 which implies that $p\leq d,$ as required.
\end{proof}

\begin{lemma}\label{centraliser}
If $g\in G,$ then $\dim_{\FF} C_{{\FF}G}(g)=|G:\langle g \rangle|.$
\end{lemma}
\begin{proof}
Let  $t_1,\dots,t_r$ be a left transversal of $\langle g \rangle$ in $G$.
Assume that $x\in C_{{\FF}G}(g).$  
As every element of $G$ can be uniquely written in the form
$t_ig^j$, we can write $x=\sum_{i,j} a_{t_ig^j} t_ig^j$, where $a_{t_ig^j}\in \FF$, and, since $xg=x$,
 we have in particular
$$a_{t_ig^j}=a_{t_ig^{j+1}}$$ for every $i$ and $j$.
 Hence $x=\sum_i b_it_i(1+g+\dots+g^{|g|-1}),$ for some $ b_i \in \FF$. Conversely,
 every $\FF$-linear combination of the elements $t_i(1+g+\dots+g^{|g|-1})$ is centralized by $g$.
 In other words
 the elements $t_i(1+g+\dots+g^{|g|-1})$, $1\leq i\leq r$, are a basis for $C_{{\FF}G}(g).$
\end{proof}

\begin{cor}\label{cor}
For every $d\in \mathbb N,$ there exists a finite 2-generated soluble group $G$
with $d_I(G)\geq d.$
\end{cor}

\begin{proof}Let $p$ be a prime number with $d\leq p$ and consider
the set $\Omega_p$ of the finite 2-generated soluble groups
whose
order is divisible by no prime smaller than $p.$ Assume by contradiction,
that $d_I(G)<d$ for every $G \in \Omega_p$
 and let $G^*$ be a group in  $\Omega_p$ such that $d_I(G^*)=\max_{G\in \Omega_p}d_I(G)$.
 Since $d_I(G^*)\le d$ and $d \le p$,
by the Theorem \ref{2-gen} there exists $H$ in $\Omega_p$ with $d_I(G^*)<d_I(H),$ and this contradicts the maximality of $d_I(G^*)$.
\end{proof}

\begin{proof}[Proof of Theorem \ref{main}]
Let $F$ be the $d$-generated free prosoluble group, with $d\geq 2.$ Assume
 that $F$ is FIG.
In particular
 $d_I(H) \le d_I(F)$ for every $2$-generated finite soluble group $H$, but this contradicts  Corollary \ref{cor}.
\end{proof}

\section{Proof of Theorem \ref{sol}}

We need, as a preliminary result, a formula for the minimal number of
generators of a $G$-module.

\begin{lemma}\label{d_G}
Let $G$ be a finite group. Assume that $A$ is a direct product
$$A=A_1^{n_1} \times \cdots \times A_r^{n_r}$$
where, for  each $i$, $A_i$ is a finite elementary abelian $p_i$-group for a prime number $p_i$,
$A_i$ is an irreducible $\FF_{p_i}G$-module and $A_i$ is not $G$-isomorphic to $A_j$ for $i \neq j$.
 Then the minimal number of elements needed to generate $A$ as $G$-module is
$$d_G(A)=\max_{i \in \{1, \ldots, r\}}   \left(\left\lceil \frac{n_{i} }{r_G(A_{i})}
\right\rceil \right),$$
where  $\lceil x \rceil$ denotes the smallest integer greater or equal to $x$.
\end{lemma}
\begin{proof}
If $J_i$ is the Jacobson radical of $\FF_{p_i}G$, then   $\FF_{p_i}G/J_i$ is semisimple and Artinian, hence we can apply the Wedderburn-Artin theorem (see e.g. \cite[Lemma 1.11, Theorems 1.14 and 3.3]{hungerford})  and we conclude that
 $A_i$  occurs  precisely $ \dim_{\End_{G}(A_i)}(A_i)=r_G(A_{i})$ times in $\FF_{p_i}G/J_i$.
Then, by \cite[Lemma 7.12]{kg},  $A$ can be generated, as $G$-module, by
$$d_G(A)=\max_ {i \in \{1, \ldots, r\}}  \left(\left\lceil \frac{n_{i} }{r_G(A_{i})}\right\rceil \right)$$
elements.
\end{proof}

\begin{prop}\label{le}
Let $G$ be a finite soluble $d$-generated group of derived length $l$.
Then $d_I(G) \le l (d-1)+1$.
\end{prop}

\begin{proof}
The proof is by induction on $l$.
If $l=1$, then $G$ is abelian and $d_I(G)=d(G)\le d= 1(d-1)+1$.

Assume $l >1$ and let $A$ be the last non-trivial term of the derived series of $G$.
Then $dl(G/A)=l-1$.
 Since $d_I(G)=d_I(G/\frat(G))$, without loss of generality we can assume $\frat(G)=1$.
Then $A$ is a direct product of complemented minimal normal subgroups of $G$ and we can write
$$A=A_1^{n_1} \times \cdots \times A_r^{n_r}$$
where each $A_i$ is an elementary abelian $p_i$-group, for a prime number $p_i$,
$A_i$ is an irreducible $\FF_{p_i}G$-module and $A_i$ is not $G$-isomorphic to $A_j$ for $i \neq j$.
Therefore by Lemma \ref{d_G}
\begin{equation}\label{d_G(A)} d_G(A)=\max_{i \in \{1, \ldots, r\}}   \left(\left\lceil \frac{n_{i} }{r_G(A_i)}
\right\rceil \right).
\end{equation}
On the other hand, by Proposition \ref{gen},
\begin{equation}\label{d}
d \ge d(G)= \max_{V} \left( \theta_G(V) +\left\lceil \frac{\delta_G(V) }{r_G(V)}\right\rceil \right)
\end{equation}
where  $V$ ranges over the set of non $G$-isomorphic complemented chief factors of $G$.
 Note that $\theta_G(A_i)=1$ for every $i$. Indeed, if we assume
 that  $A_i$ is a trivial $G$-module,
 then, as $\frat(G)=1$, we have $G=A_i \times H$ for a complement $H$ of $A_i$ in $G$. Hence $G'=H'$ and $G'$ does not contain $A_i$, 
 contradicting 
the fact that $A_i$ is a subgroup of the  last term of the derived series of $G$.

Since $n_{i} \le \delta_G(A_{i})$, by equations \ref{d_G(A)} and \ref{d}
we deduce that
$$ d \ge   \max_{i \in \{1, \ldots, r\}}     \left( 1+ \left\lceil  \frac{n_{i}  }{r_G(A_i)}\right\rceil \right)
 =1 + d_G(A)$$
 hence $d_G(A) \le d-1$.
 Let $a_{1}, \ldots , a_{d-1}$ be a set of generators for $A$ as
 $G$-module
 and let  $g_1, \ldots, g_{t}$ be invariable generators for $G$ modulo $A$ with $t=d_I(G/A)$.
Then it  is straightforward to check that
 the the elements
\[g_1, \ldots, g_{t}, a_{1}, \ldots , a_{d-1} \]
 invariably generate $G$, hence
$$d_I(G) \le t+(d-1)= d_I(G/A)+(d-1).$$
Since  $dl(G/A) = l-1$,  by inductive hypothesis we have that $$d_I(G/A) \le (l-1)(d-1)+1,$$
and we conclude that  $$d_I(G) \le (l-1)(d-1)+1 +(d-1)=l(d-1)+1,$$ as required.
\end{proof}

Denote by $dl(G)$ the derived length of a soluble group $G.$ It follows from the previous proposition, that if $G$ is a finitely generated solvable profinite group, then $d_I(G) \le dl(G)(d(G)-1)+1.$ In order to complete the proof of Theorem
\ref{sol} it suffices to prove the following result:

\begin{thm} Let $d$ be a positive integer and let $p$ be a prime number. For every positive integer $l <\frac{p-1}{d-1}+1$ there
exists a finite soluble group $G_l$ such that
\begin{itemize}
\item $p= \min{\pi(G_l)}$,
\item $dl(G_l)=l$,
\item $d(G_l)=d$,
\item $d_I(G_l)=l(d-1)+1 $.
\end{itemize}
\end{thm}
\begin{proof} We prove the theorem by induction on $l.$ If $l=1$, then
we can take $G_1=C_p^d.$ So suppose that
 a group $G_l$, with the desired properties,
 has been constructed for
 $l < \frac{p-1}{d-1}.$ As in the proof of Theorem \ref{2-gen}, if we take a prime $q$ such that the exponent of $G_l$ divides $q-1$
and we consider the field ${\FF}$ be the field of order $q,$ then
$$V:={\FF} G_l= V_1^{n_1}\oplus\cdots \oplus  V_r^{n_r}$$ where 
 the $V_j$ are absolutely  irreducible ${\FF} G$-modules no two of which are $G$-isomorphic, 
 and $n_j=\dim_{\FF} V_j.$ Consider the semidirect product $G_{l+1}=V^{d-1}\rtimes G_l$. It can be easily seen that $dl(G_{l+1})=dl(G_l)+1=l+1$ and
that $G_{l+1}$ is isomorphic to the wreath product
$C_q^{d-1} \wr G_l$ with respect to the regular permutation representation of $G_l$. In particular, by  \cite[Corollary 2.4]{andrea1}, as $C_q^{d-1}$ and $G_l$ have coprime orders, $$d(G_{l+1})=d(C_q^{d-1} \wr G_l)=\max(d(G_l),d(C_q^{d-1})+1))=d.$$
Now let $t=d_I(G_{l+1})$ and suppose that $w_1g_1,\dots,w_tg_t$, with $w_i \in V^{d-1}$ and $g_i \in G_l$, invariably
generate $G_{l+1}.$ By Corollary \ref{matrici}, for any $j\in \{1,\dots,t\}$
\begin{equation*}
(d-1)n_j\leq \sum_{i=1}^t \dim_{\FF} C_{V_j}(g_i).
\end{equation*}
As in the proof of Theorem \ref{2-gen}, this implies
\begin{equation}\label{tredi}
d-1\leq \sum_{i=1}^t\frac{\dim_{\FF} C_{{\FF}G_l}(g_i)}{|G_l|}.
\end{equation}
Notice that $g_1,\dots,g_t$ must invariably generate $G_l$ so
$t \ge d_I(G_l)=l(d-1)+1$
  and
  in particular
  we may assume $g_i\neq 1$ for every
$i\leq l(d-1)+1.$
Therefore,
by Lemma \ref{centraliser},
$$\frac{\dim_{\FF} C_{{\FF}G_l}(g_i)}{|G_l|}\leq \frac{1}{p} \quad \text { if $i\leq l(d-1)+1$}.$$
Since the trivial bound ${\dim_{\FF} C_{{\FF}G_l}(g_i)}/{|G_l|}\leq 1$ holds for all 
 $i= l(d-1)+2, \ldots , t$,
 it follows from (\ref{tredi}) that
\begin{equation*}
d-1\leq \frac{l(d-1)+1}{p}+t-l(d-1)-1
\end{equation*}
i.e.
\begin{equation*}
t\geq \left\lceil(l+1)(d-1)+1-\frac{l(d-1)+1}{p}\right\rceil.
\end{equation*}
Since we are assuming $l<\frac{p-1}{d-1},$ we have $\frac{l(d-1)+1}{p}<1$
and consequently $d_I(G_{l+1})=t\geq (l+1)(d-1)+1.$ On the other hand, since $dl(G_l)= l+1$, by Proposition \ref{le}
 we have $d_I(G_{l+1})\leq (l+1)(d-1)+1$ and therefore the equality $d_I(G_{l+1})=(l+1)(d-1)+1$ has been proved.
\end{proof}

\section{Proof of Theorem \ref{pro-supersoluble}}

\begin{prop}\label{geq}
For every $d\in \mathbb N$ there exists a finite supersoluble group $G$ such that $d(G)=d$ and
$d_I(G)\geq 2d-1.$
\end{prop}

\begin{proof}Let $K=C_2^d.$ There are $\alpha:=2^d-1$ different epimorphisms $\sigma_1,\dots,\sigma_\alpha$ from
$K$ to $C_2$ ($\sigma_i: K\to C_2$ is uniquely determined by $M_i=\ker \sigma_i,$ a $(d-1)$-dimensional subspace of $K$).
To any $i,$ there corresponds a $K$-module $V_i$ defined as follows: $V_i\cong C_3$ and $v_i^k=v_i$ if $k\in M_i$,
$v_i^k=v_i^2$ otherwise. Let $W_i=V_i^{d-1}$ and consider $G=\left(\prod_{1\leq i\leq \alpha} W_i\right)\rtimes K.$
The group $G$ is supersoluble and, by Proposition \ref{gen}, it is easy to see that   $d(G)=d$. Now assume that $g_1,\dots,g_r$ invariably generate $G.$
We write $g_i=(w_{i1},\dots,w_{i\alpha})k_i$ with $k_i \in K$ and $w_{ij}\in W_j.$ In particular $k_1,\dots,k_r$ generate $K$ and, up to reordering the elements
 $g_1,\dots,g_r,$ we can assume that the first $d$-elements  $ k_1,\dots,k_d$ are  a basis for $K.$
Let $M=\langle k_1^{-1}k_2, \dots, k_{d-1}^{-1}k_d\rangle.$ It can be easily checked that $M$ is a maximal
subgroup of $K,$ so $M=M_j$ for some $j\in\{1,\dots,\alpha\}.$ Moreover $k_i\notin M_j$ for
every $i\in\{1,\dots,d\},$ in particular $C_{V_j}(k_i)=0$ for every $i\in\{1,\dots,d\}.$ On the other hand
$w_{1j}k_1,\dots,w_{rj}k_r$ invariably generate $G$, so, by Corollary \ref{matrici},
$$d-1\leq \sum_{1\leq i \leq r}\dim_{\mathbb{F}_3} C_{V_j}(k_i)=\sum_{d+1\leq i \leq r}\dim _{\mathbb{F}_3}C_{V_j}(k_i)\leq r-d.$$
Hence $r\geq 2d-1.$
\end{proof}

 \begin{proof}[Proof of Theorem \ref{pro-supersoluble}]
Let $F$ be the free prosupersoluble group of rank $d \ge 2$.
By Proposition \ref{geq}, there exists a  finite supersoluble $d$-generated group $G$ such that
$d_I(G)\geq 2d-1.$ Hence $d_I(F) \geq 2d-1$.

To prove the converse,
since $d_I(F)=d_I(F/\frat(F))$, it suffices to consider $G=F/\frat F.$
By \cite[Proposition 3.3]{ribesprosupersolv},  $G'$ is abelian  hence $dl(G)\le 2$ and
 it follows from Proposition \ref{le} that $d_I(G) \le 2d-1$. Therefore $d_I(F) = 2d-1$.
\end{proof}

\bibliographystyle{amsplain}

\end{document}